\documentclass{siamart171218}
\usepackage{subfig}
\usepackage{forest}
\usepackage[section]{placeins}
\usepackage[export]{adjustbox}
\usepackage[title]{appendix}
\usepackage{amssymb}

\usepackage{multirow}

\newcommand{\lp}{\left(}
\newcommand{\rp}{\right)}

\newcommand{\R}{\mathbb{R}}

\newcommand{\ct}{\mathcal{T}}
\newcommand{\Tm}{\ensuremath{\nu_\text{merge}}}
\newcommand{\num}{\ensuremath{\nu_\text{merge}}}

\newcommand\Tstrut{\rule{0pt}{2.4ex}}

\newcommand{\prop}[2]{\textsf{#1}(#2)}
\newcommand{\pp}[2]{\frac{\partial #1}{\partial #2}}

\newcommand{\vect}[1]{\mathbf{#1}}

\DeclareMathOperator*{\argmin}{arg\,min}  
\newcommand{\nmax}{N}
\newcommand{\bigo}{\mathcal{O}}
\newcommand{\child}[1]{\textsf{child}$_{#1}$}

\newtheorem{thm}{Theorem}[section]

\usepackage{lipsum}
\usepackage{amsfonts}
\usepackage{graphicx}
\usepackage{epstopdf}
\usepackage{algorithmic}
\ifpdf
  \DeclareGraphicsExtensions{.eps,.pdf,.png,.jpg}
\else
  \DeclareGraphicsExtensions{.eps}
\fi

\newsiamremark{remark}{Remark}
\newsiamremark{hypothesis}{Hypothesis}
\crefname{hypothesis}{Hypothesis}{Hypotheses}
\newsiamthm{claim}{Claim}

\headers{Adaptive partition of unity}{Kevin W. Aiton, Tobin A. Driscoll}

\title{An adaptive partition of unity method for multivariate Chebyshev polynomial  approximations \thanks{Submitted to the editors May 1, 2018.
\funding{This research was supported by National Science Foundation grant DMS-1412085.}}}

\author{Kevin W. Aiton, Tobin A. Driscoll}

\usepackage{amsopn}

\begin{document}

\maketitle

\begin{abstract}
  Spectral polynomial approximation of smooth functions allows real-time manipulation of and computation with them, as in the Chebfun system. Extension of the technique to two-dimensional and three-dimensional functions on hyperrectangles has mainly focused on low-rank approximation. While this method is very effective for some functions, it is highly anisotropic and unacceptably slow for many functions of potential interest. A method based on automatic recursive domain splitting, with a partition of unity to define the global approximation, is easy to construct and manipulate. Experiments show it to be as fast as existing software for many low-rank functions, and much faster on other examples, even in serial computation. It is also much less sensitive to alignment with coordinate axes. Some steps are also taken toward approximation of functions on nonrectangular domains, by using least-squares polynomial approximations in a manner similar to Fourier extension methods, with promising results. 

\end{abstract}



\section{Introduction}
\label{sec:introduction}

A distinctive and powerful mode of scientific computation has emerged recently in which mathematical functions are represented by high-accuracy numerical analogs, which are then manipulated or analyzed numerically using a high-level toolset~\cite{Trefethen2015}. The most prominent example of this style of computing is the open-source Chebfun project~\cite{battles2004extension,Driscoll2014}. Chebfun, which is written in MATLAB, samples a given piecewise-smooth univariate function at scaled Chebyshev nodes and automatically determines a Chebyshev polynomial interpolant for the data, resulting in an approximation that is typically within a small multiple of double precision of the original function. This approximation can then be operated on and analyzed with algorithms that are fast in both the asymptotic and real-time senses. Notable operations include rootfinding, integration, optimization, solution of initial- and boundary-value problems, eigenvalues of differential and integral operators, and solution of time-dependent PDEs.

Townsend and Trefethen extended the 1D Chebfun algorithms to 2D functions over rectangles in Chebfun2~\cite{townsend2013extension,Townsend2014}, which uses low-rank approximations in an adaptive cross approximation. The construction and manipulation of 2D approximations is suitably fast for a wide range of smooth examples. Most recently, Hashemi and Trefethen created an extension of Chebfun called Chebfun3 for 3D approximations on hyperrectangles using low-rank ``slice--Tucker'' decompositions~\cite{Hashemi2017}. The range of functions that Chebfun3 can cope with in a reasonable interactive computing time is somewhat narrower than for Chebfun2, as one would expect.

One aspect of the low-rank approximations used by Chebfun2 and Chebfun3 is that they are highly anisotropic. That is, rotation of the coordinate axes can transform a rank-one or low-rank function into one with a much higher rank, greatly increasing the time required for function construction and manipulations. This issue is considered in detail in~\cite{trefethen2017cubature}. 

An alternative to Chebfun and related projects ported to other languages is sparse grid interpolation. Here one uses linear or polynomial interpolants on hierarchical Smolyak grids. Notable examples of software based on this technique are the Sparse Grid Interpolation Toolbox~\cite{Klimke2005} and the Sparse Grids Matlab Kit~\cite{Back2011}. An advantage of these packages is that they are capable of at least medium-dimensional representations on hyperrectangles. However, they seem to be less focused on high-accuracy approximation for a wide range of functions, and they are less fully featured than the Chebfun family. These methods are also highly nonisotropic.

In this work we propose decomposing a hyperrectangular domain by adaptive, recursive bisections in one dimension at a time, generalizing earlier work in one dimension \cite{Aiton2018}. The resulting subdomains are defined to be overlapping, and on each we employ simple tensor-product Chebyshev polynomial interpolants. In order to define a global smooth approximation, we use a partition of unity to blend together the subdomains. This allows the approximation to capture highly localized function features while remaining computationally tractable.

The more general problem of approximation of a function with high pointwise accuracy over a nonrectangular domain $\Omega \subset \R^d$ allows more limited global options than in the hyperrectangular case. Neither low-rank nor sparse grid approximations have any clear global generalizations to this case. Two techniques that can achieve spectral convergence for at least some such domains are radial basis functions~\cite{Fornberg2015} and Fourier extension or continuation~\cite{adcock2014resolution}, but neither has been conclusively demonstrated to operate with high speed and reliability over a large collection of domains and functions.

Our use of an adaptive decomposition allows us to approximate on such domains with great flexibility. If a base subdomain is hyperrectangular, we proceed with a tensor-product interpolation for speed, but if its intersection with the global domain is nonrectangular, we can opt for a different representation. We need not be concerned with having a very large number of degrees of freedom in any local subproblem, since further subdivision is available, so the local algorithm need not be overly sophisticated.

The adaptive construction of function approximations is based on binary trees, as explained in section~\ref{sec:construction}. In section~\ref{sec:operations} we describe fast algorithms for evaluation, arithmetic combination, differentiation, and integration of the resulting tree-based approximations. Numerical experiments over hyperrectangles in section~\ref{sec:numerical_experiments} demonstrate that the tree-based approximations exhibit far less anisotropy than do Chebfun2 and Chebfun3. Our implementation is faster than Chebfun2 and Chebfun3 on all tested examples---sometimes by orders of magnitude---except for examples of very low rank, for which all the methods are acceptably fast. In section~\ref{sec:general-domain} we describe and demonstrate approximation on nonrectangular domains using a simple linear least-squares approximation by the tensor-product Chebyshev basis. While these results are preliminary, we think they show enough promise to merit further investigation.

\section{Adaptive construction}
\label{sec:construction}

Let $\Omega = \{ \vect{x} \in \R^d: x_i \in [a_i,b_i], i=1,\ldots,d\}$ be a hyperrectangle, and suppose we wish to approximate $f:\Omega \to \R$. Our strategy is to cover $\Omega$ with overlapping subdomains, on each of which $f$ is well-approximated by a multivariate polynomial, and use a partition of unity to construct a global approximation. We defer a description of the partition of unity scheme to section~\ref{sec:operations}. In this section we describe an adaptive procedure for obtaining the overlapping domains and individual approximations over them. 

The domains are constructed from recursive bisections of $\Omega$ into nonoverlapping hyperrectangular \emph{zones}. Given a zone $\prod_{j=1}^d [\alpha_{j},\beta_{j}]$, we extend it to a larger domain $\prod_{j=1}^d [\bar{\alpha}_{j},\bar{\beta}_{j}]$ by fixing a parameter $t>0$, defining
\begin{equation}
  \label{eq:overlap}
  \delta_{j} =  \frac{\beta_{j}-\alpha_{j}}{2}(1+t),\quad j=1,\ldots,d,
\end{equation}
and then setting
\begin{equation}
  \bar{\alpha}_{j} = \max\{a_j,\beta_{j}-\delta_{j}\}, \quad \bar{\beta}_{j} = \min\{\alpha_{j}+\delta_{j},b_j\}.
  \label{eq:zone_extend}
\end{equation}
In words, the zone is extended on all sides by an amount proportional to its width in each dimension, up to the boundary of the global domain $\Omega$.

We define a binary tree $\ct$ with each node $\nu$ having the following properties:
\begin{itemize}
\item \prop{zone}{$\nu$}: zone associated with $\nu$
\item \textsf{domain}($\nu$): domain associated with $\nu$
\item \textsf{isdone}($\nu$): $n$-vector of boolean values, where $\textsf{isdone}_j$ indicates whether the domain is determined to be sufficiently resolved in the $j$th dimension
\item \child{0}($\nu$),\child{1}($\nu$): left and right subtrees of $\nu$ (empty for a leaf)
\item \textsf{splitdim}($\nu$): the dimension in which $\nu$ is split (empty for a leaf)
\end{itemize}

\noindent A leaf node has the following additional properties:
\begin{itemize}
\item \textsf{grid}($\nu$): tensor-product grid of Chebyshev 2nd-kind points mapped to \prop{domain}{$\nu$}
\item \textsf{values}($\nu$): function values at \textsf{grid}($\nu$)
\item \textsf{interpolant}($\nu$): polynomial interpolant of \textsf{values}($\nu$) on \textsf{grid}($\nu$)
\end{itemize}

\noindent If $\nu$ is a leaf, its domain is constructed by extending \prop{zone}{$\nu$} as in~(\ref{eq:zone_extend}). Otherwise, \prop{domain}{$\nu$} is the smallest hyperrectangle containing the domains of its children. 

Let $f$ be the scalar-valued function on $\Omega$ that we wish to approximate. A key task is to compute, for a given leaf node $\nu$, the polynomial \textsf{interpolant}($\nu$), and determine whether $f$ is sufficiently well approximated on \textsf{domain}($\nu$) by it. First we sample $f$ at a Chebyshev grid of size $\nmax^d$ on \textsf{domain}($\nu$). This leads to the interpolating polynomial
\begin{equation}
  \label{eq:full-interp}
  \tilde{p}(\vect{x}) = \sum_{i_1=0}^{\nmax-1} \cdots \sum_{i_d=0}^{\nmax-1}  C_{i_1,\ldots,i_d} T_{i_1}(x_1)\cdots T_{i_d}(x_d),
\end{equation}
where the coefficient array $C$ can be computed by FFT in $\bigo(\nmax^d \log \nmax)$ time~\cite{mason2002chebyshev}. Following the practice of Chebfun3t \cite{Hashemi2017}, for each $j=1,\ldots,d$, we define a scalar sequence $\gamma^{(j)}$ by summing $|C_{i_1,\ldots,i_d}|$ over all dimensions except the $j$th. To each of these sequences we apply Chebfun's {\tt StandardChop}  algorithm, which attempts to measure decay in the coefficients in a suitably robust sense~\cite{Aurentz:2017:CCS:3034774.2998442}. Let the output of {\tt StandardChop} for sequence $\gamma^{(j)}$ be $n_j$; this is the degree that {\tt StandardChop} deems to be sufficient for resolution at a user-set tolerance.  If $n_j<\nmax$ we say that the function is resolved in dimension $j$ on $\nu$. If $f$ is resolved in all dimensions on $\nu$, then we truncate the interpolant sums in~(\ref{eq:full-interp}) at the degrees $n_j$ and store the samples of $f$ on the corresponding smaller tensor-product grid. 

\begin{algorithm}
\caption{refine($\nu$,$f$,$\nmax$,$t$)}
\label{alg:refine}
\begin{algorithmic}
  \IF{$\nu$ is a leaf}
    \STATE Sample $f$ on \textsf{grid}($\nu$)
    \STATE Determine chopping degrees $n_1,\ldots,n_d$
    \FOR{each $j$ with \textsf{isdone}($\nu$)$_j=$ FALSE}
      \IF{$n_j<\nmax$}
        \STATE \textsf{isdone}($\nu$)$_j$ := TRUE
      \ELSE
        \STATE split($\nu$,$j$,$t$)
      \ENDIF
    \ENDFOR
    \IF{all \textsf{isdone}($\nu$) are TRUE}
      \STATE Truncate~(\ref{eq:full-interp}) at degrees $n_1,\ldots,n_d$ to define \textsf{grid}($\nu$), \textsf{values}($\nu$), \textsf{interpolant}($\nu$)
    \ELSE
      \STATE refine($\nu$,$f$,$\nmax$,$t$)
    \ENDIF
  \ELSE
    \STATE refine(\child{0}($\nu$),$f$,$\nmax$,$t$)      
    \STATE refine(\child{1}($\nu$),$f$,$\nmax$,$t$)
  \ENDIF
\end{algorithmic}
\end{algorithm}

Algorithm~\ref{alg:refine}  describes a recursive adaptation procedure for building the binary tree $\ct$, beginning with a root node whose zone and domain are both the original hyperrectangle $\Omega$. For a non-leaf input, the algorithm is simply called recursively on the children. For an input node that is currently a leaf of the tree, the function $f$ is sampled, and chopping is used in each unfinished dimension to determine whether sufficient resolution has been achieved. Each dimension that is deemed to be resolved is marked as finished. If all dimensions are found to be finished, then the interpolant is chopped to the minimum necessary length in each dimension, and the node will remain a leaf. Otherwise, the node is split in all unfinished dimensions using Algorithm~\ref{alg:split}, and Algorithm~\ref{alg:refine} is applied recursively. Note that the descendants of a splitting inherit the \textsf{isdone} property that marks which dimensions have been finished, so no future splits are possible in such dimensions within this branch. 

\begin{algorithm}
\caption{split($\nu$,$j$,$t$)}
\label{alg:split}
\begin{algorithmic}
\IF{$\nu$ is a leaf}
\STATE \textsf{splitdim}($\nu$)=$j$
\STATE Define new nodes $\nu_0$, $\nu_1$
\STATE $[a_1,b_1],[a_2,b_2],\dots,[a_n,b_n]$ be the subintervals from \prop{zone}{$\nu$}
\STATE Let $m:= \frac{b_j+a_j}{2}$
\STATE Let \textsf{zone}($\nu_0$) $:= [a_1,b_1] \times \dots \times [a_{j-1},b_{j-1}] \times [a_{j},m] \times [a_{j+1},b_{j+1}] \times \dots \times [a_{d},b_{d}] $
\STATE Let \textsf{zone}($\nu_1$) $:= [a_1,b_1] \times \dots \times [a_{j-1},b_{j-1}] \times [m,b_{j}] \times [a_{j+1},b_{j+1}] \times \dots \times [a_{d},b_{d}] $
\FOR{$k=0,1$}
\STATE Define \textsf{domain}($\nu_k$) from \textsf{zone}($\nu_k$) with parameter $t$ as in~(\ref{eq:zone_extend})
\STATE Define \textsf{grid}($\nu_k$) as Chebyshev tensor-product grid of size $\nmax^d$ in \textsf{domain}($\nu_k$)
\STATE Let \textsf{isdone}($\nu_k$):= \textsf{isdone}($\nu$)
\ENDFOR
\ELSE
\STATE split(\child{0}($\nu$),$k$,$t$)
\STATE split(\child{1}($\nu$),$k$,$t$)
\ENDIF
\end{algorithmic}
\end{algorithm}

\section{Computations with the tree representation}
\label{sec:operations}

The procedure of the preceding section constructs a binary tree $\ct$ whose leaves each hold an accurate representation of $f$ over a subdomain. These subdomains overlap, and constructing a global partition of unity approximation from them is straightforward.

Define the $C^\infty$ function
\begin{align}
  \label{eq:basic-cinf}
  \psi_0(x) = \begin{cases}
    \exp \lp  1 - \frac{1}{1-x^2}\rp & |x| \leq 1, \\
    0 & |x| > 1,
  \end{cases}
\end{align}
and let
\begin{equation}
  \label{eq:affine}
  \ell(x;a,b) = 2\frac{x-a}{b-a} - 1
\end{equation}
be the affine map from $[a,b]$ to $[-1,1]$. Suppose $\nu$ is a leaf of $\ct$ with domain $\Omega_\nu = \prod [\bar{\alpha}_j,\bar{\beta}_j]$. Then we can define the smoothed-indicator or bump function
\begin{equation}
  \label{eq:bump-functions}
  \psi_\nu(\vect{x}) = \prod_{j=1}^d \psi_0\bigl( \ell(x_j;\bar{\alpha}_{j},\bar{\beta}_{j}) \bigr).
\end{equation}
Next we use Shepard's method~\cite{wendland2004scattered} to define a partition of unity $\{w_\nu (\vect{x})\}$, indexed by the leaves of $\ct$:
\begin{equation}
  \label{eq:pu-weight}
  w_\nu(\vect{x}) = \frac{\psi_\nu(\vect{x})}{\displaystyle \sum_{\mu\in \text{leaves}(\ct)} \psi_\mu(\vect{x})}.
\end{equation}
We have $\sum_{\nu\in \text{leaves}(\ct)} w_\nu(\vect{x}) = 1$, which makes $\{w_\nu (\vect{x})\}$ a partition of unity. This implies that  $w_\nu(\vect{x})=1$ for any $\vect{x}$ that lies in $\nu$ and no other patches. Thus if we assume that weight functions are supported only in their respective domains, smoothness of the partition of unity functions requires overlap between neighboring patches.

 Let $s_\nu$ be the polynomial interpolant of $f$ over the domain of node $\nu$. Then the global partition of unity approximant is
\begin{equation}
  s(\vect{x}) = \sum_{\nu\in \text{leaves}(\ct)} w_\nu(\vect{x})s_\nu(\vect{x}).
  \label{eq:pu-approx}	
\end{equation}
Despite consisting of separate local approximations from a partitioned domain, the global approximation (\ref{eq:pu-approx}) remains infinitely smooth while avoiding explicit global matching constraints. This permits rapid (in principle, beyond all orders) convergence to smooth functions, as well as generating continuous derivative approximations~\cite{wendland2004scattered}.

 While this approximation is globally continuous it is still local in some sense. As an example, in Figure~\ref{overlap_plot} we plot the overlapping patches on the domain of a patch $\nu$ for the partition of unity approximation of $\arctan(3(y^2+x))$ (which can be seen in Figure~\ref{TANFUN1}). We see that in the interior of the patch that the approximation (\ref{eq:pu-approx}) would consist only of the polynomial approximation $s_\nu(\vect{x})$, and in the overlap would blend neighboring approximations with the partition of unity.

\begin{figure}
\centering
\includegraphics[scale = 0.5]{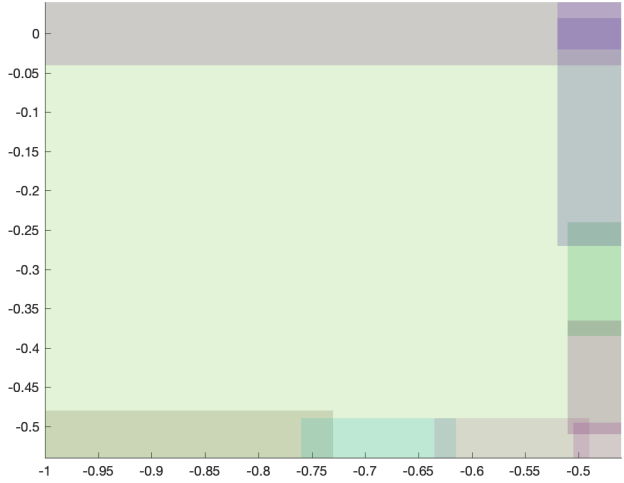}
\caption{Plot of the subdomains formed from the partition of unity method for $\arctan(3(y^2+x))$ on a local patch with domain $[-1 , -0.46] \times [-0.54 , 0.04]$.}
\label{overlap_plot}
\end{figure}

Next we describe efficient algorithms using the tree representation of the global approximant to perform common numerical operations such as evaluation at points, basic binary arithmetic operations on functions, differentiation, and integration.

\subsection{Evaluation}

Note that (\ref{eq:pu-weight})--~(\ref{eq:pu-approx}) can be rearranged into
\begin{equation}
  s(\vect{x}) = \sum_{\nu\in \text{leaves}(\ct)} \frac{s_\nu(\vect{x}) \psi_\nu(\vect{x})}{\displaystyle \sum_{\mu\in \text{leaves}(\ct)} \psi_\mu(\vect{x})}
  = \frac{\displaystyle\sum_{\nu\in \text{leaves}(\ct)} s_\nu(\vect{x}) \psi_\nu(\vect{x})}{\displaystyle \sum_{\mu\in \text{leaves}(\ct)} \psi_\mu(\vect{x})}.
  \label{eq:pu-alt}
\end{equation}
This formula suggests a recursive approach to evaluating the numerator and denominator, presented in Algorithm~\ref{alg:numden}. Using it, only leaves containing $\vect{x}$ and their ancestors are ever visited. A similar approach was described in~\cite{tobor2006reconstructing}.

\begin{algorithm}
\caption{[$S$,$P$]=numden($\nu$,$\vect{x}$)}
\label{alg:numden}
\begin{algorithmic}
\STATE $S=0$, $P=0$
\IF{$\nu$ is a leaf}
\STATE $S=\psi_\nu(\vect{x})$
\STATE $P = S \cdot \textsf{interpolant}(\nu)(\vect{x})$
\ELSE
\FOR{$k=0,1$}
\IF{$\vect{x} \in \text{\textsf{domain}(\child{k}}(\nu))$}
\STATE $[S_k,P_k]$ = numden(\child{k}($\nu$),$\vect{x}$) 
\STATE $S = S + S_k$
\STATE $P = P + P_k$
\ENDIF
\ENDFOR
\ENDIF
\end{algorithmic}
\end{algorithm}

Algorithm~\ref{alg:numden} can easily be vectorized to evaluate $s(\vect{x})$ at multiple points, by recursively calling each leaf with all values of $\vect{x}$ that lie within its domain. In the particular case when the evaluation is to be done at all points in a Cartesian grid, it is worth noting that the leaf-level interpolant in~(\ref{eq:full-interp}) can be evaluated by a process that yields significant speedup over a naive approach. As a notationally streamlined example, say that the desired values of $\vect{x}$ are $(\xi_{j_1},\ldots,\xi_{j_d})$, where each $j_k$ is drawn from $\{1,\ldots,M\}$, and that the array of polynomial coefficients is of full size $O(N^d)$. Express~(\ref{eq:full-interp}) as
\begin{multline}
  \label{eq:grid-eval}
  \sum_{i_1=0}^{\nmax-1} \cdots \sum_{i_d=0}^{\nmax-1}  C_{i_1,\ldots,i_d} T_{i_1}(\xi_{j_1})\cdots T_{i_d}(\xi_{j_d})
  \\ = \sum_{i_1=0}^{\nmax-1} T_{i_1}(\xi_{j_1}) \sum_{i_2=0}^{\nmax-1} T_{i_2}(\xi_{j_2}) \cdots \sum_{i_d=0}^{\nmax-1}  C_{i_1,\ldots,i_d} T_{i_d}(\xi_{j_d}).
\end{multline}
The innermost sum yields $N^{d-1}M$ unique values, each taking $\bigo(N)$ time to compute. At the next level there are $N^{d-2}M^2$ values, and so on, finally leading to the computation of all $M^d$ interpolant values. This takes $\bigo(MN(M+N)^{d-1})$ operations, as opposed to $\bigo(M^dN^d)$ when done naively.

\subsection{Binary arithmetic operations}

Suppose we have two approximations ${s}_1(\vect{x})$, ${s}_2(\vect{x})$, represented by trees $\ct_1$ and $\ct_2$ respectively, and we want to construct a tree approximation for ${s}_1 \circ {s}_2$, where $\circ$ is one of the operators $+$, $-$, $\times$, or $\div$. If $\ct_1$ and $\ct_2$ have identical tree structures, then it is straightforward to operate leafwise on the polynomial approximations. In the cases of multiplication and division, the resulting tree may have to be refined further using Algorithm~\ref{alg:split}, since these operations typically result in polynomials of degree greater than the operands. 

If the trees $\ct_1$ and $\ct_2$ are not structurally identical, we are free to use Algorithm~\ref{alg:split} to construct an approximation by sampling values of ${s}_1 \circ {s}_2$. However, the tree of ${s}_1 \circ {s}_2$ likely shares refinement structure with both $\ct_1$ and $\ct_2$. For example, Figure~\ref{zone_tan} shows the refined zones of the trees for $\arctan(100(x^2+y))$, $\arctan(100(x+y^2))$, and their sum. Thus in practice we merge the trees $\ct_1$ and $\ct_2$ using Algorithm~\ref{alg:merge}, presented in Appendix~\ref{sec:merge}. The merged tree, whose leaves contain sampled values of the result, may then be refined further if chopping tests then reveal that the result is not fully resolved. 

\begin{figure}
\centering
\subfloat[Zone plot of $f_1(x,y)$]{
\includegraphics[scale = 0.3]{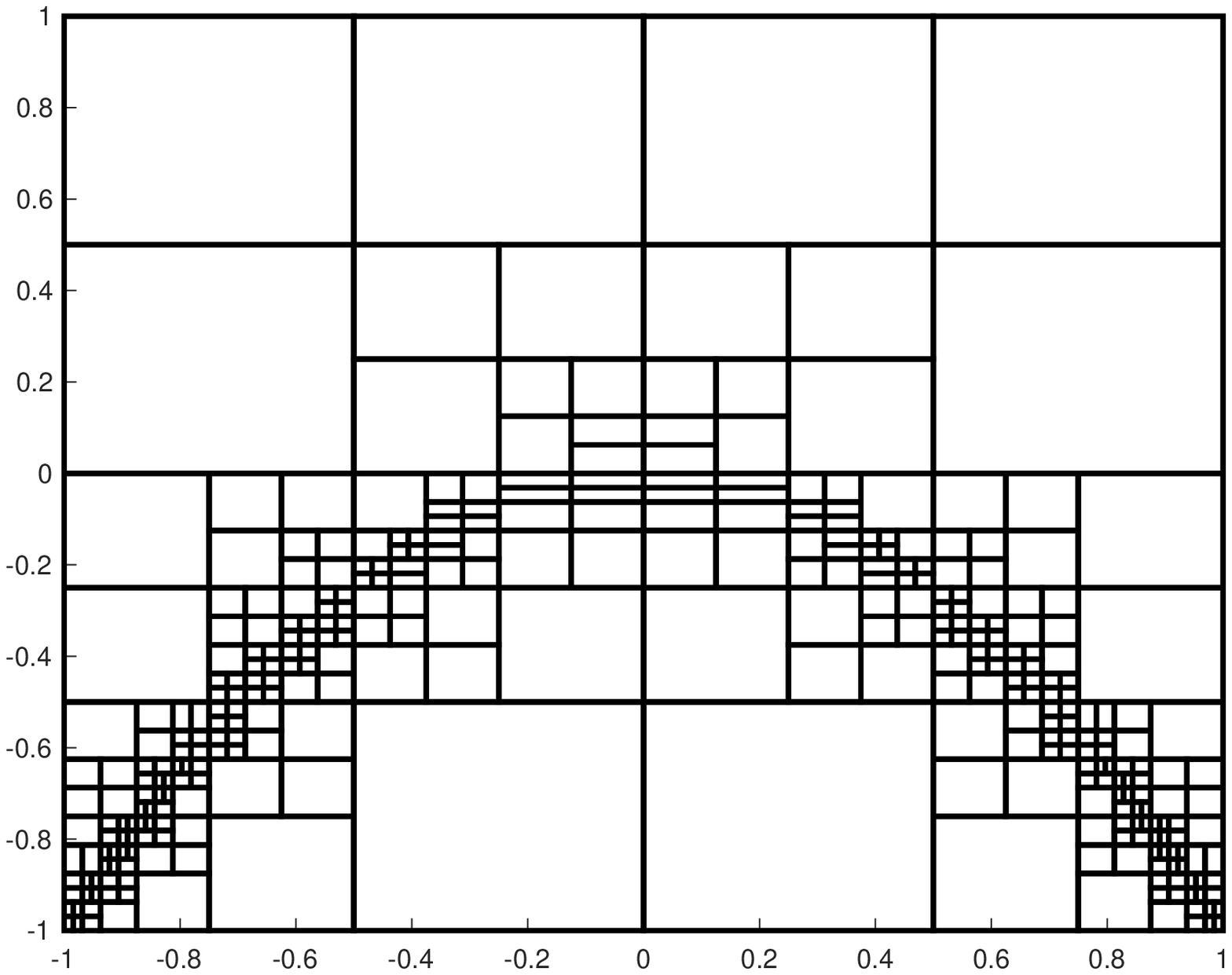}
   \label{zone_tan_a}
 }
\subfloat[Zone plot of $f_2(x,y)$]{
\includegraphics[scale = 0.3]{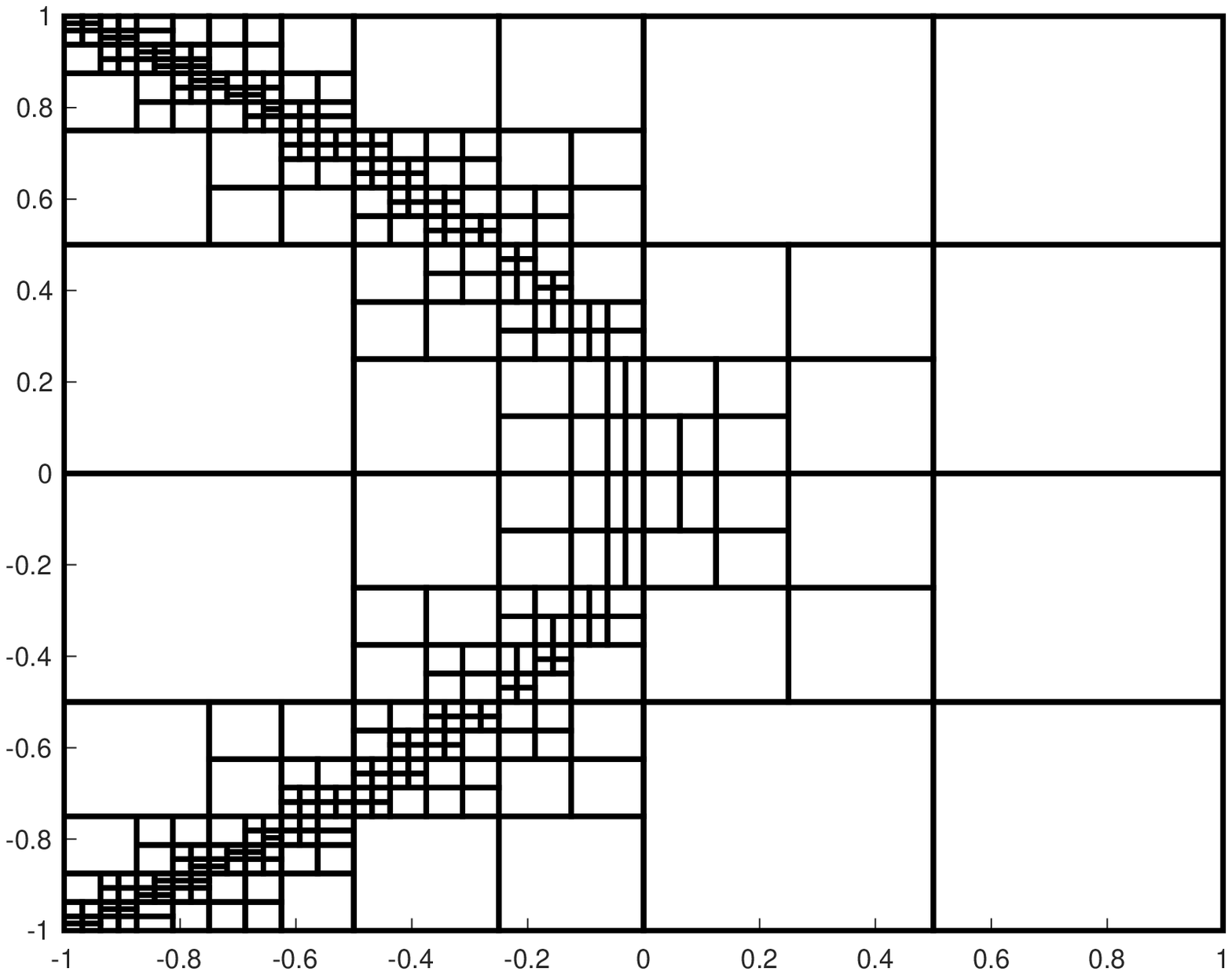}
   \label{zone_tan_b}
 }
 
 \subfloat[Zone plot of $f_1(x,y)+f_2(x,y)$]{
\includegraphics[scale = 0.3]{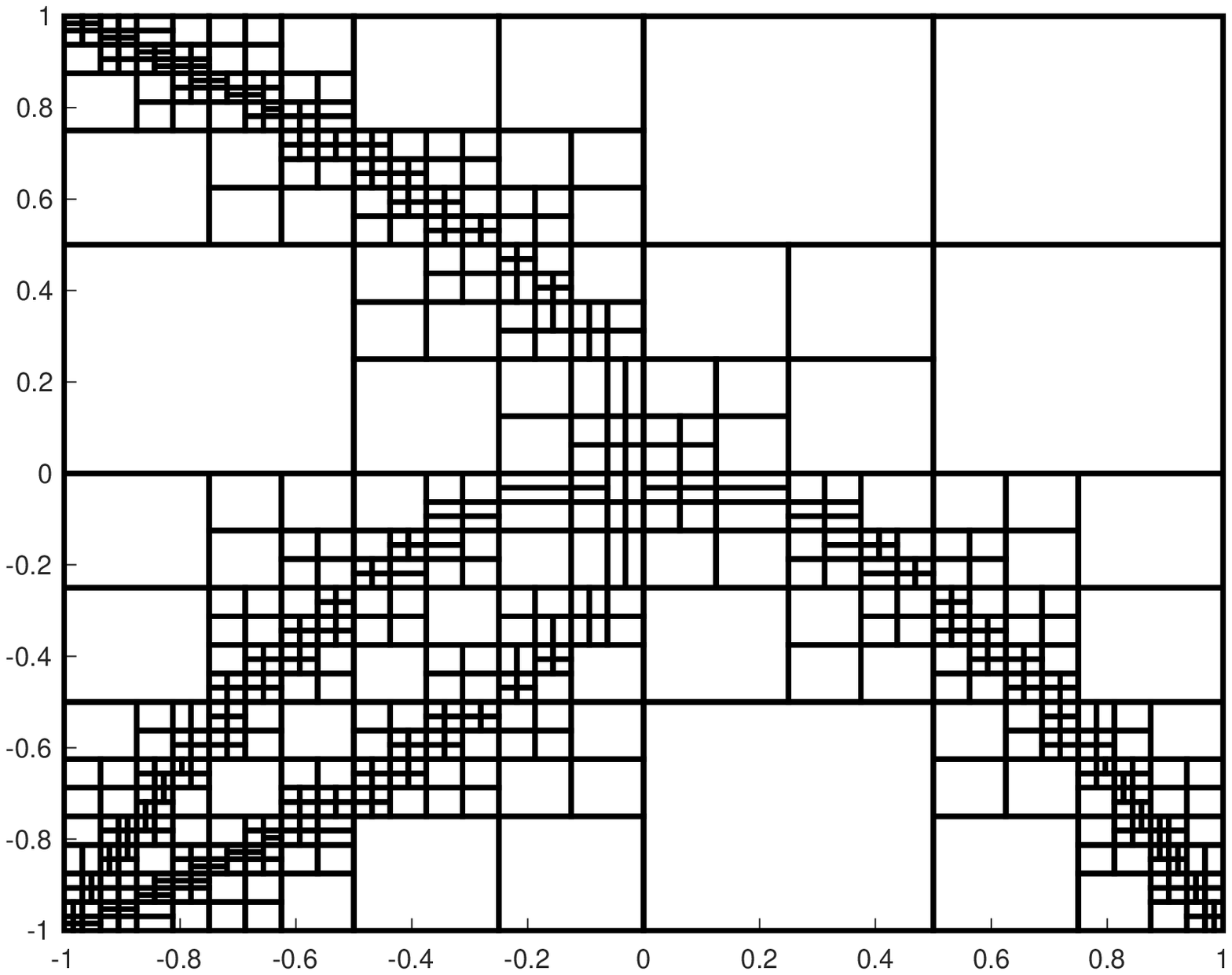}
   \label{zone_tan_c}
 }
 \caption{Zone plots for $f_1(x,y)$,$f_2(x,y)$ and $f_1(x,y)+f_2(x,y)$.}
\label{zone_tan}
\end{figure}

\subsection{Differentiation}

Differentiation of the global approximant~(\ref{eq:pu-approx}) results in two groups of terms:
\begin{equation*}
  \pp{}{x_j} s(\vect{x})=\sum_{\nu\in \text{leaves}(\ct)} w_\nu(\vect{x}) \pp{}{x_j} s_\nu(\vect{x})+\sum_{\nu\in \text{leaves}(\ct)} s_\nu(\vect{x}) \pp{}{x_j} w_\nu(\vect{x}).
\end{equation*}
The first sum is a partition of unity approximation of leafwise differentiated interpolants. That is, we simply apply standard spectral differentiation to the data stored in the leaves of $\ct$. Although it may seem surprising at first, we can define the desired derivative approximation solely in terms of this first sum, and neglect the second with little penalty.

\begin{thm}
  Define
  \begin{equation}
    \label{eq:pu-deriv}
    s^{(j)}(\vect{x})=\sum_{\nu\in \text{leaves}(\ct)} w_\nu(\vect{x}) \pp{}{x_j} s_\nu(\vect{x}).
  \end{equation}
  Then for all $\vect{x}\in\Omega$,
  \begin{equation}
    \label{eq:pu-deriv-error}
    \left|s^{(j)}(\vect{x})-\pp{f}{x_j}(\vect{x})\right| \le \sum_{\vect{x} \in \prop{domain}{\nu}} w_\nu(\vect{x}) \left| \pp{s_\nu}{x_j}(\vect{x}) - \pp{f}{x_j} (\vect{x})\right|. 
  \end{equation}
\end{thm}
\begin{proof}
  By the partition of unity property,
  \[
    s^{(j)}(\vect{x})-\pp{f}{x_j}(\vect{x}) = \sum_{\nu\in \text{leaves}(\ct)} w_\nu(\vect{x}) \left[\pp{}{x_j} s_\nu(\vect{x}) - \pp{f}{x_j}(\vect{x})\right].
  \]
  The result follows because $w_\nu(\vect{x})=0$ if $\vect{x}\notin \prop{domain}{\nu}$.
\end{proof}

Hence if $\vect{x}$ is not in an overlap region, the error in the global derivative approximation $s^{(j)}$ is the same as for the local approximant. Otherwise, it is bounded---pessimistically, since the weights are positive and sum to unity pointwise---by the sum of errors in all the contributing approximants. Since no point can be in more than $2^d$ subdomains (and then only near a meeting of hyperrectangle corners), we feel this error is acceptable in two and three dimensions.  


\subsection{Integration}

The simplest and seemingly most efficient approach to integrating over the domain is to do so piecewise over the nonoverlapping zones,
\begin{equation}
  \label{eq:integration}
  \int_{\Omega} f(\vect{x}) d\vect{x} = \sum_{\nu\in \text{leaves}(\ct)} \int_{\text{\textsf{zone}($\nu$)}} f(\vect{x}) d\vect{x}.
\end{equation}
Since the leaf interpolants are defined natively over the overlapping domains, they must be resampled at Chebyshev grids on the zones, after which Clenshaw-Curtis quadrature is applied.

\section{Numerical experiments}
\label{sec:numerical_experiments}

All the following experiments were performed on a computer with a 2.6 GHz Intel Core i5 processor in version 2017a of MATLAB. Our code, which uses a serial object-oriented recursive implementation of the algorithms, is available for download.\footnote{\url{https://github.com/kevinwaiton/PUchebfun}} Comparisons to Chebfun2 and Chebfun3 were done using Chebfun version 5.5.0. We also tried to use the Sparse Grid Interpolation Toolbox~\cite{Klimke2005}, but on all the examples we were unable to get it close to our desired error tolerances within its hard-coded limits on sparse grid depth. 

\subsection{2D experiments}

We first test the 2D functions $\log(1+(x^2+y^4)/10^{-5})$, $\arctan((x+y^2)/10^{-2})$, $\frac{10^{-4}}{(10^{-4}+x^2)(10^{-4}+y^2)}$, Franke's function \cite{franke1979critical}, the smooth functions from the Genz family test package \cite{genz1987package}, and the ``peg'' examples from~\cite{trefethen2017cubature}. For each function we record the time of construction, the time to evaluate on a $200\times 200$ grid, and the max observed error on this grid. Table~\ref{tab:timing2D} shows the results for the new method. For the low-rank test cases, the methods are comparable, with neither showing a consistent advantage; most importantly, both methods are fast enough for interactive computing. In the tests of higher-rank functions, the tree-based method exhibits a clear, sometimes dramatic, advantage in construction time. Moreover, the tree method remains fast enough for interactive computing even as the total number of nodes exceeds 1.6 million.  We present plots of the functions and adaptively generated subdomains for the first three test functions in Figures~\ref{TANFUN1}-\ref{rungeFUN1}.

\begin{table}[p]

  \begin{tabular}{c|c|c|c|c|c}
    \multirow{2}{*}{Function} & \multirow{2}{*}{Alg.} & \multirow{2}{*}{Error} & Build & Eval & Points /  \\
                            & &  & time & time & Rank  \\ \hline
    \multirow{2}{*}{$\log(1+\frac{x_1^2+x_2^4}{10^{-5}})$} &  T & 1.16$\times 10^{-15}$ &	0.525 &	0.1235 & 69800 \Tstrut  \\
                              & C & 1.14$\times 10^{-6}$ & 2.30 & 0.10 & 30 \\ \hline
    \multirow{2}{*}{$\arctan(\frac{x_1+x_2^2}{10^{-2}})$} & T & 1.83$\times 10^{-14}$ &	2.241 & 0.3590	& 917515 \Tstrut\\
                              & C & 7.09$\times 10^{-12}$ & 150 & 5.0 & 816 \\ \hline
    \multirow{2}{*}{$\frac{10^{-4}}{(10^{-4}+x_1^2)(10^{-4}+x_2^2)}$} & T & 1.86$\times 10^{-15}$ & 0.606 & 0.0728 & 	117056 \Tstrut\\
                              & C & 5.44$\times 10^{-15}$ & 0.049 & 0.0037 & 1 \\ \hline
    \multirow{2}{*}{franke} & T & 1.33$\times 10{-15}$ & 	0.061 &	0.0069 &	9270 \Tstrut\\
                              & C & 1.33$\times 10^{-15}$ & 0.020 & 0.0024 & 4 \\ \hline
    \multirow{2}{*}{$\cos(u_1\pi + \sum_{i=1}^2 a_i x_i)$} & T & 23.00$\times 10^{-15}$ &	0.007 &	0.0012 & 	972 \Tstrut\\
                              & C & 4.47$\times 10^{-14}$ & 0.016 & 0.0020 & 2 \\ \hline
    \multirow{2}{*}{$\prod_{i=1}^2 (a_i^{-2}+(x_i-u_i)^2)^{-1}$} & T & 2.01$\times 10^{-15}$ &	0.063 &	0.0099 & 21232 \Tstrut\\
                              & C & 1.59$\times 10^{-12}$ & 0.020 & 0.0022 & 1 \\ \hline
    \multirow{2}{*}{$(1+\sum_{i=1}^2 a_i x_i)^{-3}$} & T & 3.33$\times 10^{-16}$ &	0.006 &	0.0004 &	 25 \Tstrut\\
                              & C & 2.27$\times 10^{-12}$ & 0.012 & 0.0021 & 4 \\ \hline
    \multirow{2}{*}{$\exp(-\sum_{i=1}^2 a_i^2 (x_i-u_i)^2)$} & T & 7.77$\times 10{-16}$ & 0.005 &	0.0012 & 	1862 \Tstrut\\
                              & C & 4.44$\times 10^{-16}$ & 0.015 & 0.0022 & 1 \\  \hline
    \multirow{2}{*}{square peg} & T & 2.22$\times 10^{-15}$ & 0.126 &	0.0264 & 111188 \Tstrut\\
                              & C & 1.22 $\times 10^{-15}$ &	0.023 &	0.0012 & 	1 \\ \hline
    \multirow{2}{*}{tilted peg} & T & 2.00$\times 10^{-15}$ &	0.214 &	0.0375 & 117544 \Tstrut\\
                              & C & 7.68$\times 10^{-14}$ & 0.265 & 	0.0181 & 100
                              
  \end{tabular}
  \caption{Observed error and wall-clock times for the tree-based (T) and Chebfun2 (C) algorithms with target tolerance $10^{-16}$ and $\nmax=129$. Build time is for constructing the approximation object, and eval time for evaluating an approximant on a 200x200 uniform grid (all times in seconds). Also shown: for the tree-based method, the total number of stored sampled function values, and for Chebfun2, the numerically determined rank of the function. Here $u=[0.75,0.25]$ and $a=[5,10]$.}
  \label{tab:timing2D}
\end{table}


\begin{figure}
  \centering
  \subfloat[$\arctan \lp (x+y^2)/0.01 \rp$]{
    \includegraphics[scale = 0.34]{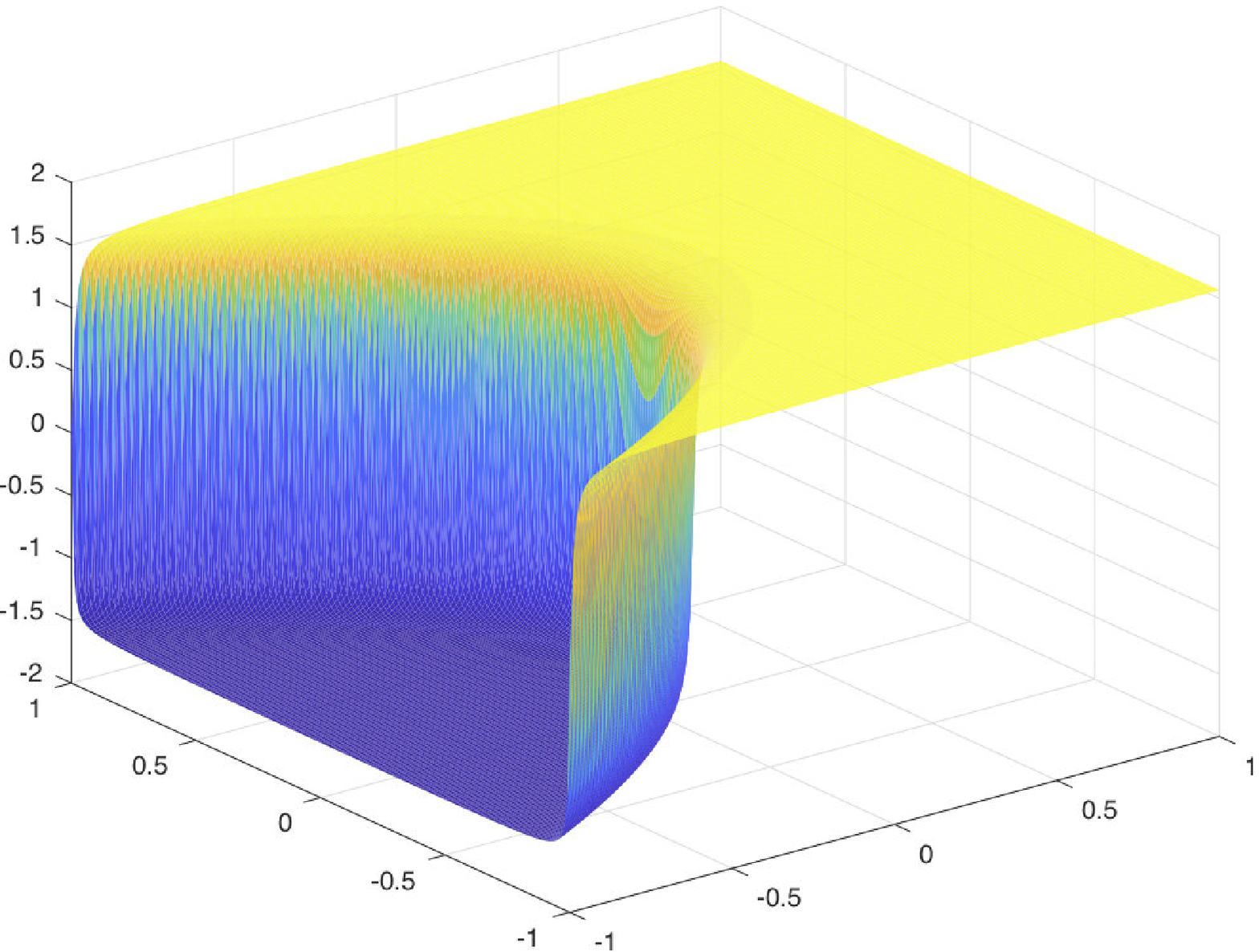}
    \label{tanfunplot}
  }
  \subfloat[Overlapping subdomains]{
    \includegraphics[scale = 0.34]{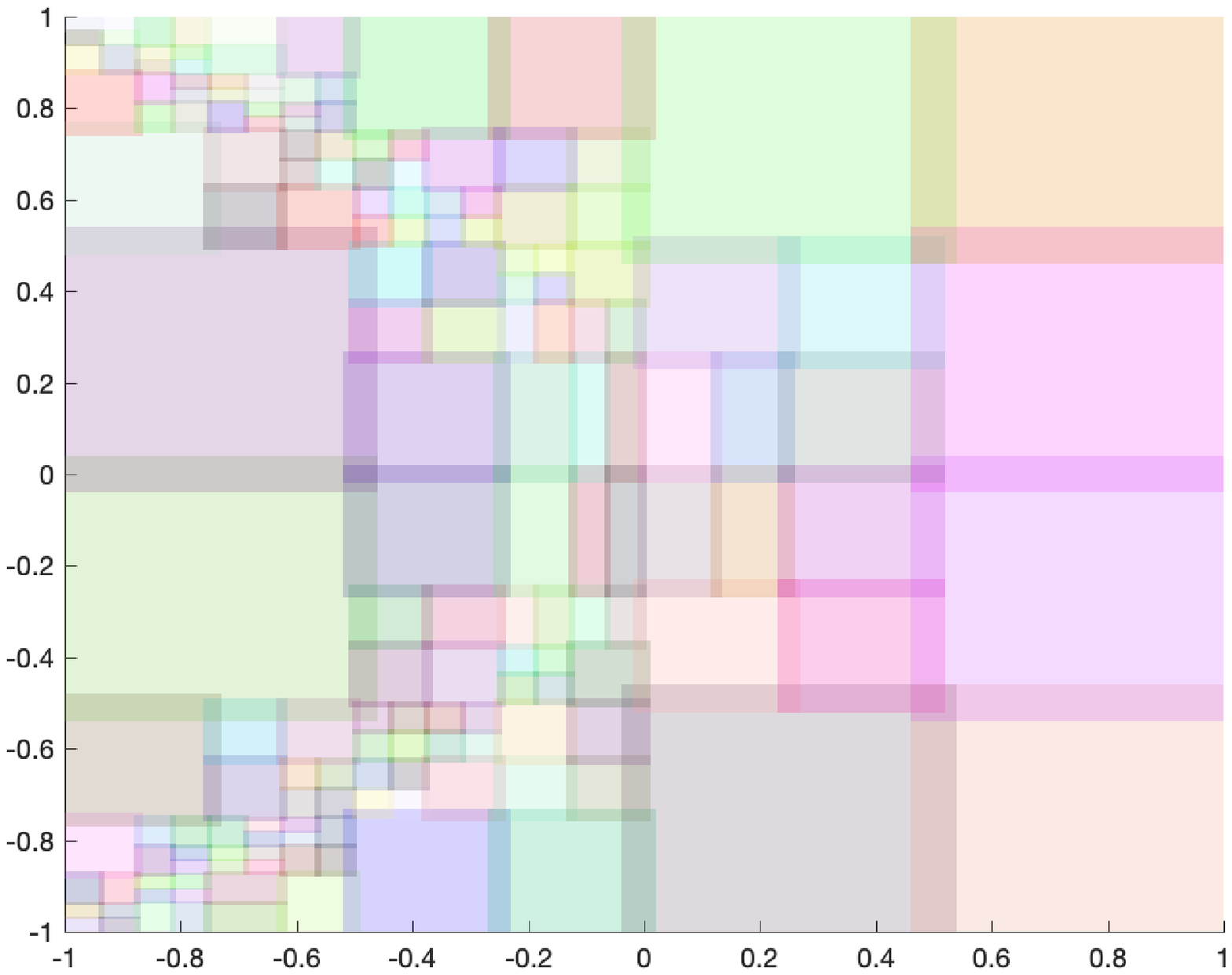}
    \label{tanfundomains}
  }
  \caption{Overlapping subdomains constructed by the adaptive tree method for a function with a nonlinear ``cliff.''}
  \label{TANFUN1}
\end{figure}

\begin{figure}
  \centering
  \subfloat[$\frac{10^{-4}}{(10^{-4}+x^2)(10^{-4}+y^2)}$]{
    \includegraphics[scale = 0.34]{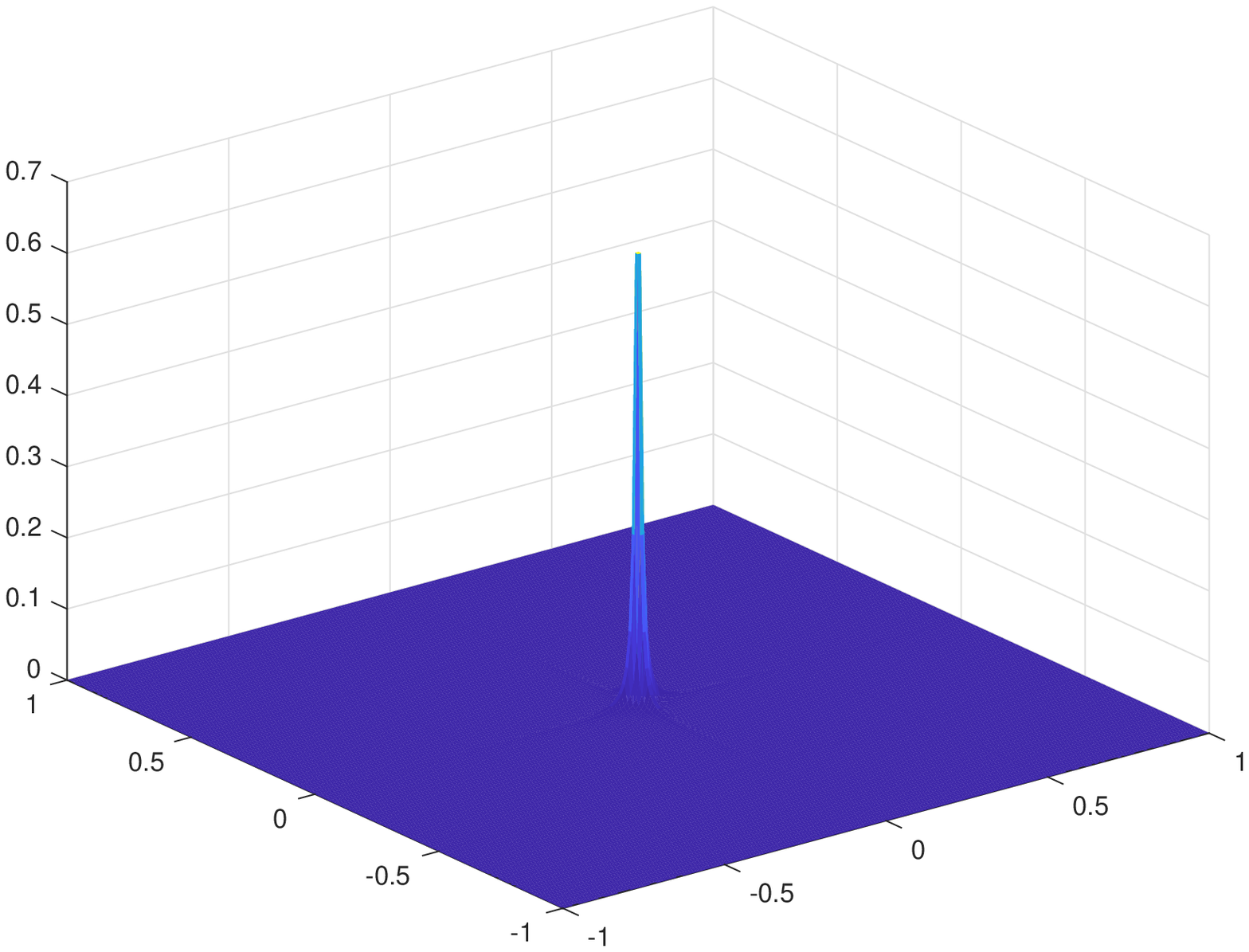}
    \label{rungefunplot}
  }
  \subfloat[Overlapping subdomains]{
    \includegraphics[scale = 0.34]{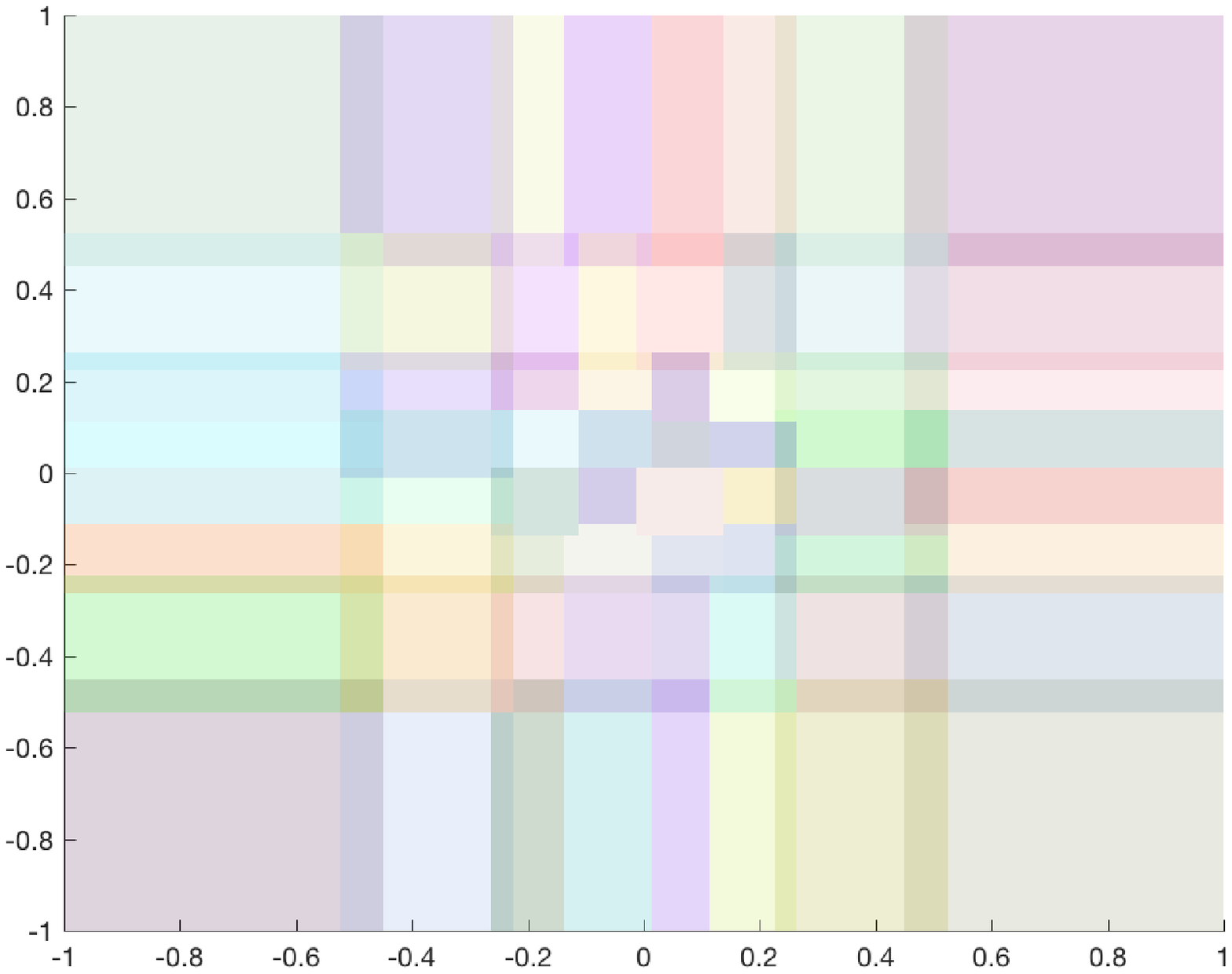}
    \label{rungefundomains}
  }
  \caption{Overlapping subdomains constructed by the adaptive tree method for a function with a sharp spike.}
  \label{rungeFUN1}
\end{figure}

One important aspect of low-rank approximation is that it is inherently nonisotropic. Consider the 2D ``plane wave bump'' 
\begin{equation}
f(x,y)=\arctan(250(\cos(t)x+\sin(t)y))
\label{rotate_func_2D}	
\end{equation}
whose normal makes an angle $t$ with the positive $x$-axis. We compare the construction times of our method to Chebfun2 for $t \in [0,\pi/4]$ in Figure~\ref{tan_rotate_2D}. We observe the execution time of Chebfun2 varying over nearly three orders of magnitude. While our method is also responsive to the angle of the wave, the variation in time is about half an order of magnitude, and our codes are faster in all but the rank-one case $t=0$ (for which both methods are fast).

\begin{figure}
\centering
\includegraphics[scale = 0.5]{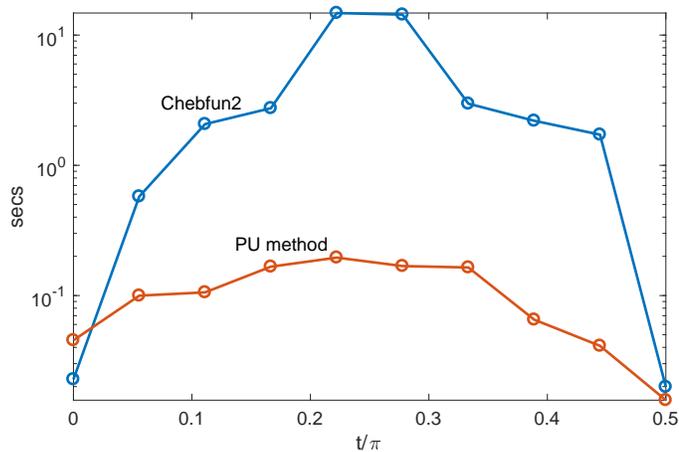}
\caption{Comparison of construction times for $\arctan(250(\cos(t)x+\sin(t)y))$ for $t \in [0,\pi/4]$.}
\label{tan_rotate_2D}
\end{figure}

Our next experiment is to add and multiply the rank-one function $\arctan(250x)$ to the plane wave in~(\ref{rotate_func_2D}). The construction time results are compared for $t \in [0,\pi/2]$ in Figure~\ref{TAN_ADD_MULT}. Here the dependence of Chebfun2 on the angle is less severe than in the simple construction, though it is still more pronounced than for our method. More importantly, the absolute numbers for addition in particular with Chebfun2 would probably be considered unacceptable for interactive computation, while our method takes one second at most. 

\begin{figure}
\centering
\includegraphics[scale = 0.5]{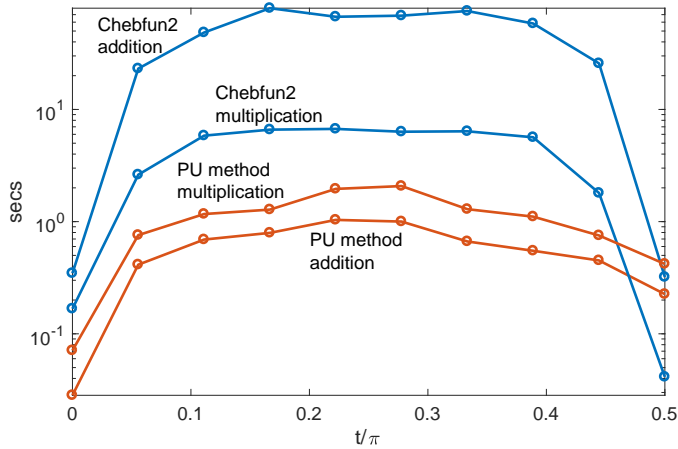}
\caption{Comparison of execution times for multiplication and addition of $\arctan(250x)$ with $\arctan(250(\cos(t)x+\sin(t)y))$ for $t \in [0,\pi/4]$.}
\label{TAN_ADD_MULT}
\end{figure}


\subsection{3D experiments}

We next test the 3D functions $1/(\cosh(5(x+y+z)))^2$, $\arctan(5(x+y)+z)$, and 3D versions of the smooth functions from the Genz family test package. Table~\ref{tab:timing3D} shows the construction time, the time taken to evaluate on a $200\times 200 \times 200$ grid, and the max error on this grid. We observe dramatic construction timing differences in every case: Chebfun3 outperforms the tree-based method for low-Tucker-rank functions, while for the two higher-rank cases, the tree-based method is the clear winner. Chebfun3 performance is more extreme in both senses, while the tree-based method is more consistent across these examples. Chebfun3 is also faster for evaluation overall, even in high-rank cases, though the evaluation times are typically far less than the construction times.

\begin{table}[p]
  \begin{tabular}{c|c|c|c|c|c}
    \multirow{2}{*}{Function} & \multirow{2}{*}{Alg.} & \multirow{2}{*}{Error} & Build & Eval & Points /  \\
 & &  & time & time &Rank  \\ \hline
    \multirow{2}{*}{$\cos(u_1\pi + \sum_{i=1}^3 a_i x_i)$} & T & $3.16 \times 10^{-14}$ &	2.958 &	0.240 &	561495 \Tstrut\\
    & C & $2.19 \times 10^{-14}$ 	& 0.460 &	0.036 &	2 \\ \hline
    \multirow{2}{*}{$\prod_{i=1}^3 (a_i^{-2}+(x_i-u_i)^2)^{-1}$} & T & $2.37 \times 10^{-15}$ 	& 9.917 & 0.764 & 7751626 \Tstrut\\
    & C & $2.63 \times 10^{-15}$	 &  0.148 &	0.030 &	1 \\ \hline
    \multirow{2}{*}{$(1+\sum_{i=1}^3 a_i x_i)^{-4}$} & T & $5.58 10^{-16}$ & 0.351 &	0.020 &	216 \Tstrut\\
    & C & $8.93 \times 10^{-16}$ &	0.174 &	0.021 & 5\\ \hline
    \multirow{2}{*}{$\exp(-\sum_{i=1}^2 a_i^2 (x_i-u_i)^2)$} & T & $1.45\times 10^{-15}$ &	0.566 &	0.097 &	293305 \Tstrut  \\
    & C & $7.80 \times 10^{-16}$ &	0.066 &	0.018 & 1 \\ \hline
    \multirow{2}{*}{$1/(\cosh(5(x+y+z)))^2$} & T & $2.00\times 10^{-15}$	 & 4.337	 & 0.325 &	3450018 \Tstrut\\
    & C & $3.66 \times 10^{-13}$ 	& 74.446	 &0.050 &	93 \\ \hline
    \multirow{2}{*}{$\arctan(5(x+y)+z)$} & T & $1.95 \times 10^{-15}$ &	0.758	& 0.145	& 1132326 \Tstrut\\
    & C & $3.17\times 10^{-13}$	& 75.313	 & 0.033	& 110
  \end{tabular}
  \caption{Observed error and wall-clock times for the tree-based (T) and Chebfun3 (C) algorithms with target tolerance $10^{-16}$ and $\nmax=65$. Build time is for constructing the approximation object, and eval time for evaluating an approximant on a $200^3$ uniform grid (all times in seconds). Also shown: for the tree-based method, the total number of stored sampled function values, and for Chebfun3, the numerically determined rank of the function. Here $u=[0.75,0.25,-0.75]$ and $a=[25,25,25]$.} 
  \label{tab:timing3D}
\end{table}

We repeat our experiment testing the importance of axes alignment using the function
\begin{equation}
\arctan(5(\sin(p)\cos(t)x+\sin(p)\sin(t)y+\cos(p)z))
\end{equation}
for $p,t \in [0,\pi/4]$. Timing results can be seen in Figure~\ref{fig:tan3D}. As in 2D, the Chebfun low-rank technique shows wide variation depending on the angles, and a large region of long times. The tree-based method is much less sensitive and faster (by as much as two orders of magnitude) except for the purely axes-aligned cases. 

\begin{figure}
  \centering
  \includegraphics[width=\textwidth]{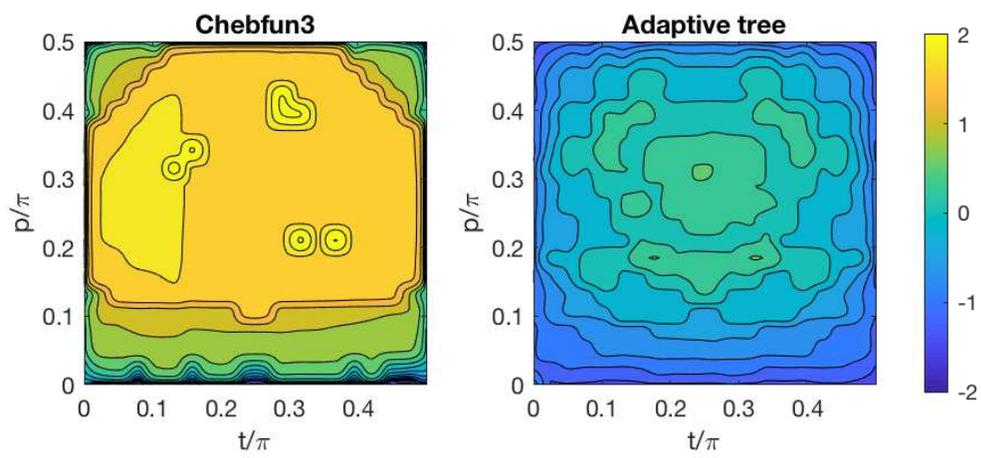}
  \caption{Construction time comparison for the 3D function $\arctan(5(\sin(p)\cos(t)x+\sin(p)\sin(t)y+\cos(p)z))$, with varying angles. Colors and contours correspond to the base-10 log of execution time in seconds.}
  \label{fig:tan3D}
\end{figure}

\section{Extension to nonrectangular domains}
\label{sec:general-domain}

We now consider approximation over a nonrectangular domain $\Omega\subset \R^d$. In our construction, a leaf node $\nu$ whose domain $\Omega_\nu$ lies entirely within $\Omega$ can be treated as before. However, if $\Omega_\nu\cap \Omega \subsetneq \Omega_\nu$, we use a different approximation technique on $\nu$. The refinement criteria of Algorithm~\ref{alg:refine} are also modified for this situation.

\subsection{Algorithm modifications}

On a leaf whose domain extends outside of $\Omega$, we again use a tensor-product Chebyshev polynomial as in~(\ref{eq:full-interp}), but choose its coefficient array $C$ by satisfying a discrete least squares criterion:
\begin{equation}
  \argmin_{C} \sum_{i=1}^{P} \lp f(\vect{x}_i) -  \tilde{p}(\vect{x}_i) \rp^2,
\end{equation}
where $\Xi = \{\vect{x}_i\}_{i=1}^{P} \subset \Omega_\nu\cap \Omega$ is a point set in the ``active'' part of the leaf's domain, $\Omega_\nu\cap \Omega$. In practice we can form a matrix $A$ whose columns are evaluations of each basis function at the points in $\Xi$, leading to a standard $P\times N^d$ linear least squares problem. We choose $\Xi$ as the part of the standard $(2N)^d$-sized Chebyshev grid lying inside $\Omega$. 

This technique resembles Fourier extension or continuation techniques~\cite{adcock2014resolution,huybrechs2010fourier}, so we refer to it as a \emph{Chebyshev extension approximation}. Unlike the Fourier case, however, there is no real domain extension involved; rather one constrains the usual multivariate polynomial only over part of its usual tensor-product domain. The condition number of $A$ in the Fourier extension case has been shown to increase exponentially with the degree of the approximation \cite{adcock2014numerical}, because the collection of functions spanning the approximation space is a \emph{frame} rather than a basis. We see the same phenomenon with Chebyshev extension; essentially, constraining the polynomial over only part of the hypercube leaves it underdetermined. To cope with the numerical rank deficiency of $A$, we rely on the basic least-squares solution computed by the MATLAB backslash. We found this to be as good as or better than the pseudoinverse with a truncated SVD. 

We modify Algorithm~\ref{alg:split} so that when a domain is split, the resulting zones of the children are shrunk if possible to just contact the boundary of $\Omega$. (An exception is the shared interface between the newly created children, which is fixed.) This helps to keep a substantial proportion of a leaf's domain within $\Omega$.

We also modify how refinement decisions are made and executed in Algorithm~\ref{alg:refine}, for a subtle reason. The original algorithm is able to exploit the very different resolution requirements for a function such as, say, $xT_{60}(y)$, by testing for sufficient resolution in each dimension independently and splitting accordingly. We find experimentally that if the function is like this over $\Omega$, the extension of it to the unconstrained part of the leaf node's domain has uniform resolution requirements in all variables. Therefore, we use a simpler refinement process: if the norm of the least-squares residual (normalized by $\sqrt{P}$) is not acceptably small, we split in all dimensions successively. In effect, the approximation becomes a quadtree or octree within those nodes that do not lie entirely within $\Omega$.

\subsection{Numerical experiments}

We chose the test functions 
\begin{equation}
  \label{eq:testfun-gen}
  \begin{aligned}
    g_1& =\exp(x+y), & g_2&=\dfrac{1}{((x-1.1)^2)+(y-1.1)^2)^2}, \\
    g_3 &=\cos(24x-32y)\sin(21x-28y), & g_4&=\arctan(3(x^2+y)).
  \end{aligned}
\end{equation}
We approximated each function on each of three domains: the unit disk, the diamond $|x|+|y|\le 1$, and the double astroid seen in Figure~\ref{star_plot}. The initial box (root of the approximation tree) was chosen to tightly enclose the given domain. For each test we set $N=17$ and the target tolerance to $10^{-10}$. We timed both the adaptive construction and the evaluation on a $200\times 200$ grid, and recorded the max error as in the previous section. In each case, we choose initial box to fit the domain as tightly as possible. These results can be seen in Table~\ref{table_general}. The resulting approximation of $g_4$ on the double astroid is shown in Figure~\ref{star_plot}, along with the adaptively found subdomains. 

When the function is smooth or contains localized features, we find that the method is both efficient and highly accurate; in the smoothest case of $g_1$, a global multivariate least-squares polynomial is sufficient. Only for $g_3$, which requires uniformly fine resolution throughout the domains, is there a construction time longer than a few seconds. The Fourier extension methods described in~\cite{matthysen2017function} are implemented in Julia, making a direct quantitative comparisons difficult, but based on the orders of magnitude of the results reported there, we feel confident that our results for these examples are superior.



\begin{figure}
\centering
\subfloat[Plot of $\arctan(3(y^2+x))$.]{
\includegraphics[scale = 0.34]{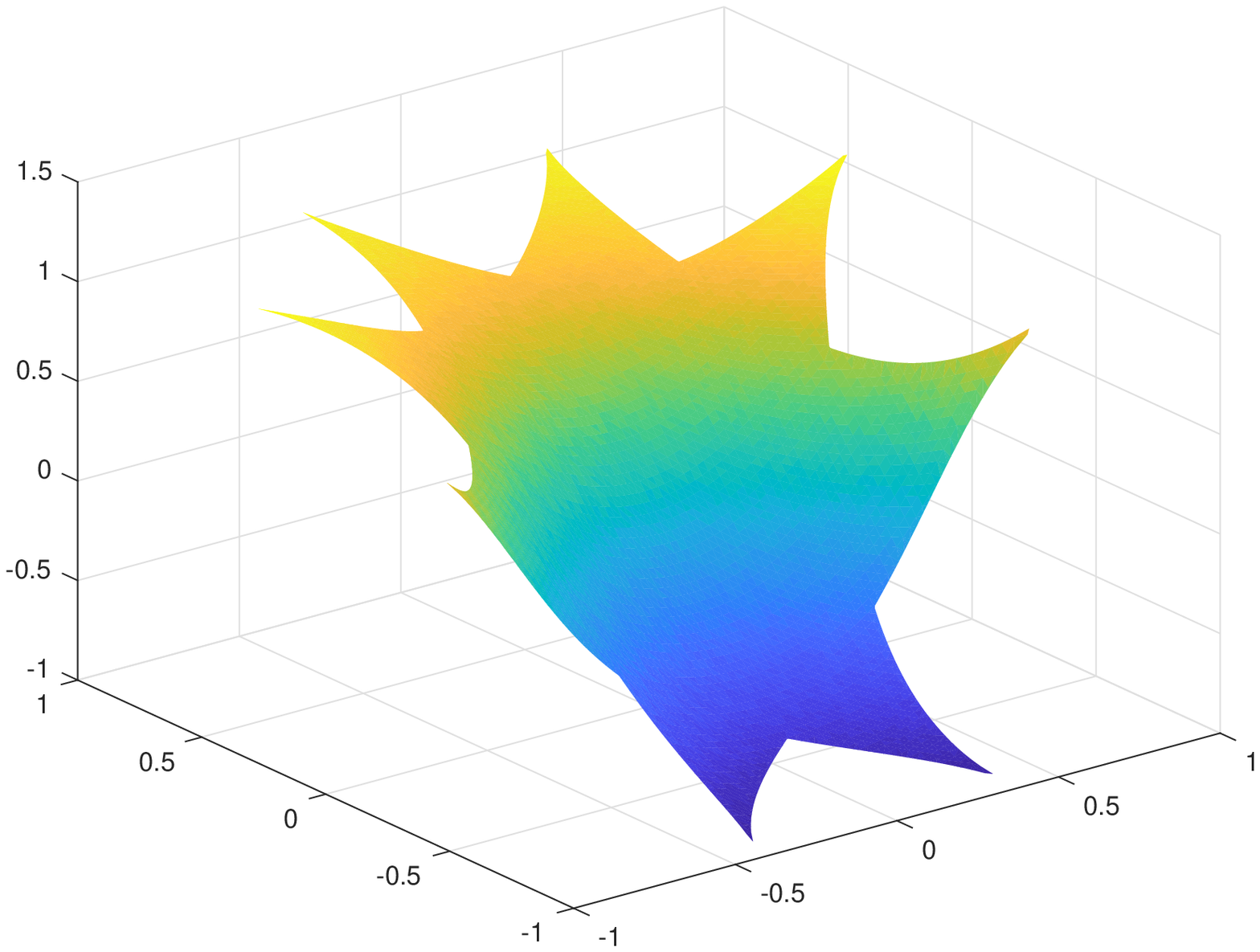}
   \label{star_plota}
 }
\subfloat[Plot of subdomains.]{
\includegraphics[scale = 0.34]{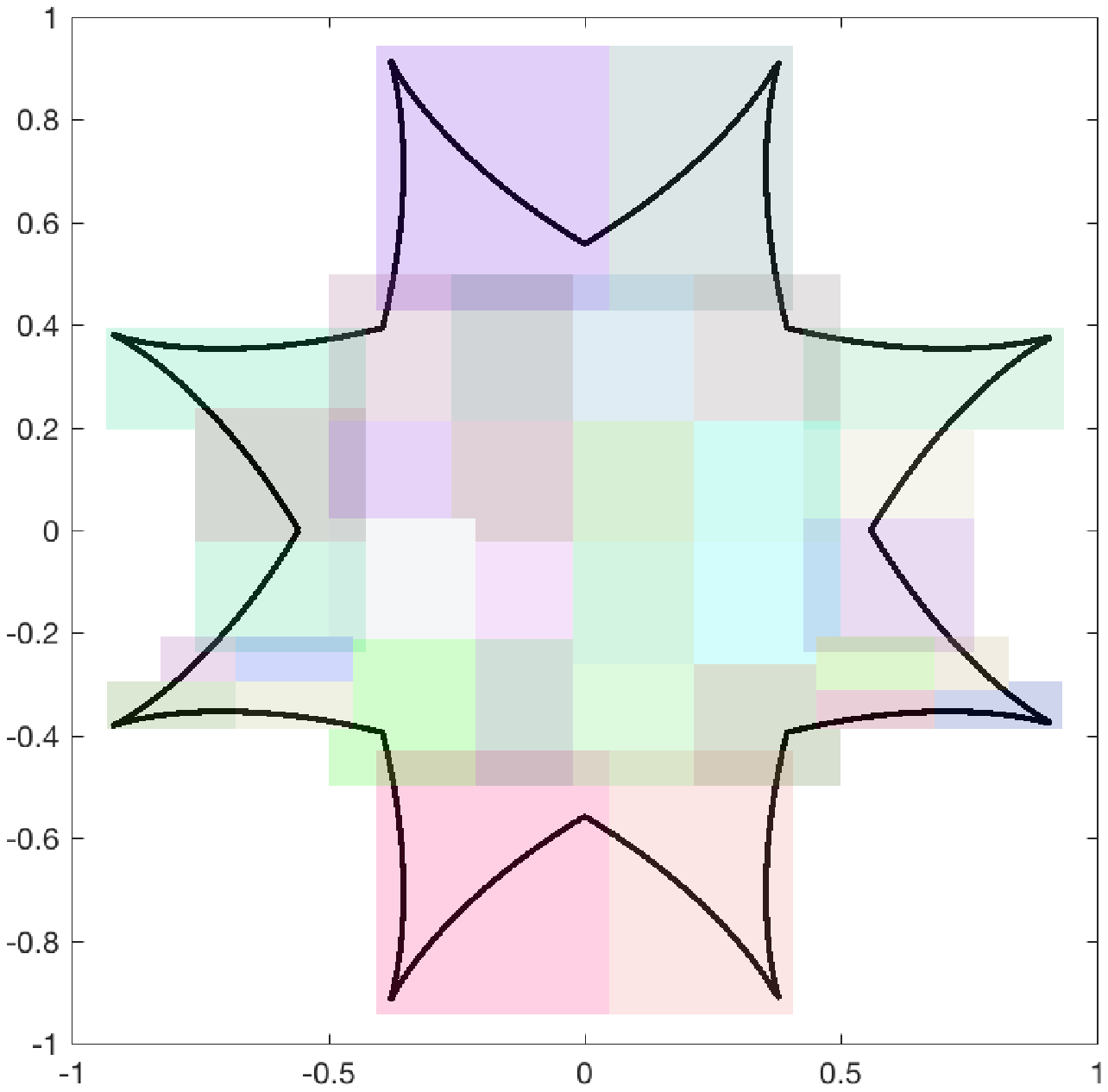}
   \label{star_plotb}
 }
\caption{Plot of $\arctan(3(y^2+x))$ and the subdomains formed from the partition of unity method. The error in this approximation was found to be about $10^{-11}$.}
\label{star_plot}
\end{figure}

\begin{table}
\begin{tabular}{c|c|c|c|c|c}
function & domain & error & construct time & interp time & points \\ [5pt] \hline
\multirow{3}{*}{ $g_1$ } & disk & 5.44E-15 & 1.369 & 0.012 & 289 \\
& diamond & 2.06E-11 & 0.040 & 0.002 & 289 \\
& astroid & 2.01E-08 & 0.071 & 0.001 & 289 \\ \hline
\multirow{3}{*}{ $g_2$ } & disk & 2.40E-10 & 2.558 & 0.117 & 3757 \\
& diamond & 2.40E-11 & 0.406 & 0.012 & 2023 \\
& astroid & 2.14E-10 & 1.511 & 0.023 & 4624 \\ \hline
\multirow{3}{*}{ $g_3$ } & disk & 4.44E-11 & 11.305 & 1.500 & 245650 \\
& diamond & 2.35E-11 & 10.894 & 0.854 & 178020 \\
& astroid & 1.67E-10 & 28.072 & 0.836 & 153780 \\ \hline
\multirow{3}{*}{ $g_4$ } & disk & 7.49E-11 & 1.866 & 0.059 & 12138 \\
& diamond & 1.45E-11 & 1.536 & 0.053 & 9826 \\
& astroid & 1.09E-11 & 3.221 & 0.049 & 9826 \\
\end{tabular}
  \caption{Observed error and wall-clock times for the adaptive tree method to approximate the functions given in~(\ref{eq:testfun-gen}) on three different 2D domains. Also shown is the total number of sampled function values stored over all the leaves of each final tree.}
  \label{table_general}
\end{table}

\section{Concluding remarks}
\label{sec:conclusion}

For functions over hyperrectangles of uncorrelated variables or that otherwise are well-aligned with coordinate axes, low-rank and sparse-grid approximations can be expected to be highly performant. We have demonstrated an alternative adaptive approach that, in two or three dimensions, typically performs very well on such functions but is far less dependent on that property. Our method sacrifices the use of a single global representation that could achieve true spectral convergence, but in practice we are able to use a partition of unity to construct a smooth, global approximation of very high accuracy in a wide range of examples. 

The adaptive domain decomposition offers some other potential advantages we have not yet exploited, but are studying. It offers a built-in parallelism for function construction and evaluation. It allows efficient updating of function values locally, rather than globally, over the domain. Finally, it has a built-in preconditioning strategy, based on additive Schwarz methods, for the solution of partial differential equations.

By replacing tensor-product interpolation on the leaves with a simple least-squares approximation using the same multivariate polynomials, we have been able to demonstrate at least reasonable performance in approximation over nonrectangular domains. Further investigation is required to better understand the least-squares approximation process, optimize adaptive strategies, and find efficient algorithms for merging trees and operations such as integration.

\bibliographystyle{siamplain}
\bibliography{PU_BIB}

\begin{thebibliography}{10}

\bibitem{adcock2014resolution}
{\sc B.~Adcock and D.~Huybrechs}, {\em On the resolution power of {F}ourier
  extensions for oscillatory functions}, Journal of Computational and Applied
  Mathematics, 260 (2014), pp.~312--336.

\bibitem{adcock2014numerical}
{\sc B.~Adcock, D.~Huybrechs, and J.~Mart{\'\i}n-Vaquero}, {\em On the
  numerical stability of {F}ourier extensions}, Foundations of Computational
  Mathematics, 14 (2014), pp.~635--687.

\bibitem{Aiton2018}
{\sc K.~W. Aiton and T.~A. Driscoll}, {\em An adaptive partition of unity
  method for {C}hebyshev polynomial interpolation}, {SIAM} Journal on
  Scientific Computing, 40 (2018), pp.~A251--A265,
  \url{https://doi.org/10.1137/17m112052x}.

\bibitem{Aurentz:2017:CCS:3034774.2998442}
{\sc J.~L. Aurentz and L.~N. Trefethen}, {\em Chopping a {C}hebyshev series},
  ACM Trans. Math. Softw., 43 (2017), pp.~33:1--33:21,
  \url{https://doi.org/10.1145/2998442},
  \url{http://doi.acm.org/10.1145/2998442}.

\bibitem{Back2011}
{\sc J.~B\"ack, F.~Nobile, L.~Tamellini, and R.~Tempone}, {\em Stochastic
  spectral {G}alerkin and collocation methods for {PDE}s with random
  coefficients: a numerical comparison}, in Spectral and High Order Methods for
  Partial Differential Equations, J.~Hesthaven and E.~Ronquist, eds., vol.~76
  of Lecture Notes in Computational Science and Engineering, Springer, 2011,
  pp.~43--62.
\newblock Selected papers from the ICOSAHOM '09 conference, June 22-26,
  Trondheim, Norway.

\bibitem{battles2004extension}
{\sc Z.~Battles and L.~N. Trefethen}, {\em An extension of {MATLAB} to
  continuous functions and operators}, SIAM J. Sci. Comp., 25 (2004),
  pp.~1743--1770.

\bibitem{Driscoll2014}
{\em {C}hebfun {G}uide}, Pafnuty Publications, 2014.

\bibitem{Fornberg2015}
{\sc B.~Fornberg and N.~Flyer}, {\em A primer on radial basis functions with
  applications to the geosciences}, SIAM, 2015.

\bibitem{franke1979critical}
{\sc R.~Franke}, {\em A critical comparison of some methods for interpolation
  of scattered data}, tech. report, Naval Postgraduate School, Monterey,
  California, 1979.

\bibitem{genz1987package}
{\sc A.~Genz}, {\em A package for testing multiple integration subroutines}, in
  Numerical Integration, Springer, 1987, pp.~337--340.

\bibitem{Hashemi2017}
{\sc B.~Hashemi and L.~N. Trefethen}, {\em Chebfun in three dimensions}, {SIAM}
  Journal on Scientific Computing, 39 (2017), pp.~C341--C363,
  \url{https://doi.org/10.1137/16m1083803}.

\bibitem{huybrechs2010fourier}
{\sc D.~Huybrechs}, {\em On the {F}ourier extension of nonperiodic functions},
  SIAM Journal on Numerical Analysis, 47 (2010), pp.~4326--4355.

\bibitem{Klimke2005}
{\sc A.~Klimke and B.~Wohlmuth}, {\em Algorithm 847: {S}pinterp: piecewise
  multilinear hierarchical sparse grid interpolation in {MATLAB}}, {ACM}
  Transactions on Mathematical Software, 31 (2005), pp.~561--579,
  \url{https://doi.org/10.1145/1114268.1114275}.

\bibitem{mason2002chebyshev}
{\sc J.~C. Mason and D.~C. Handscomb}, {\em Chebyshev Polynomials}, CRC Press,
  2002.

\bibitem{matthysen2017function}
{\sc R.~Matthysen and D.~Huybrechs}, {\em Function approximation on arbitrary
  domains using {F}ourier extension frames}, arXiv preprint arXiv:1706.04848,
  (2017).

\bibitem{tobor2006reconstructing}
{\sc I.~Tobor, P.~Reuter, and C.~Schlick}, {\em Reconstructing multi-scale
  variational partition of unity implicit surfaces with attributes}, Graphical
  Models, 68 (2006), pp.~25--41.

\bibitem{townsend2013extension}
{\sc A.~Townsend and L.~N. Trefethen}, {\em An extension of {C}hebfun to two
  dimensions}, SIAM Journal on Scientific Computing, 35 (2013), pp.~C495--C518.

\bibitem{Townsend2014}
{\sc A.~Townsend and L.~N. Trefethen}, {\em Continuous analogues of matrix
  factorizations}, Proceedings of the Royal Society A: Mathematical, Physical
  and Engineering Sciences, 471 (2014), pp.~20140585--20140585,
  \url{https://doi.org/10.1098/rspa.2014.0585}.

\bibitem{Trefethen2015}
{\sc L.~N. Trefethen}, {\em Computing numerically with functions instead of
  numbers}, Communications of the {ACM}, 58 (2015), pp.~91--97,
  \url{https://doi.org/10.1145/2814847}.

\bibitem{trefethen2017cubature}
{\sc L.~N. Trefethen}, {\em Cubature, approximation, and isotropy in the
  hypercube}, SIAM Review, 59 (2017), pp.~469--491.

\bibitem{wendland2004scattered}
{\sc H.~Wendland}, {\em Scattered Data Approximation}, Cambridge University
  Press, 2004.

\end{thebibliography}

\begin{appendices}

\section{Merging trees}
\label{sec:merge}

Algorithm~\ref{alg:merge} describes a recursive method for merging two trees $\ct_1$ and $\ct_2$, representing functions $f_1$ and $f_2$, into a tree representation for $f_1\circ f_2$, with $\circ$ as $+$, $-$, $\times$, or $\div$. The input arguments to the algorithm are the operation, corresponding nodes of $\ct_1$, $\ct_2$, and the merged tree, and the number $r$, which is the dimension that was most recently split in the merged tree. Initially the algorithm is called with root nodes representing the entire original domain, and $r=0$. 

We assume an important relationship among the input nodes. Suppose that \prop{zone}{$\nu_k$}=$\prod_{j=1}^d [\alpha_{kj},\beta_{kj}]$ for $k=1,2$, and that \prop{zone}{$\Tm$}=$\prod_{j=1}^d [A_{j},B_{j}]$. Then we require for $k=1,2$ that
\begin{equation}
  \label{eq:merge-zones}
  [a_{kj},b_{kj}]=[A_j,B_j] \quad \text{for all $j$ having} \quad   \prop{isdone}{\nu_k}_j=
  \text{FALSE}.
\end{equation}
This is trivially true at the root level. The significance of this requirement is that it allows us to avoid ambiguity about what the zone of $\Tm$ should be after a new split in, say, dimension $j$. Since only an uncompleted dimension can be split, the zone of the children of $\Tm$ after splitting will be identical to that of whichever (or both) of the $\nu_k$ requires refinement in dimension $j$.

For example, suppose the zones of $\nu_1$ and $\nu_2$ are $[-1,0]\times[-1,1]$ and $[-1,1]\times[0,1]$, respectively, and \prop{zone}{$\Tm$}=$[-1,0]\times[0,1]$. It is clear that we can interpolate from $\nu_1$ and $\nu_2$ onto $\Tm$. It is also clear that we can further split in $x$ in $\nu_1$, and in $y$ in $\nu_2$. But if we were to split $\nu_2$ in $x$, one of the children would have zone $[0,1]\times[0,1]$, which is inaccessible to $\nu_1$. 

Consider the general recursive call. If both $\nu_1$ and $\nu_2$ are leaves, then we simply evaluate the result of operating on their interpolants to get the values on $\Tm$. If exactly one of $\nu_1$ and $\nu_2$ is a leaf, then we split $\Tm$ the same way as the non-leaf and recurse into the resulting children; property~(\ref{eq:merge-zones}) trivially remains true in these calls. If both  $\nu_1$ and $\nu_2$ are non-leaves, and they both split in the same dimension, then we can split $\Tm$ in that dimension and recurse, and the zones will continue to match as in~(\ref{eq:merge-zones}).

The only remaining case is that $\nu_1$ and $\nu_2$ are each split, but in different dimensions. In this case we have to use information about how the splittings are constructed in Algorithm~\ref{alg:refine}. Recall that each unresolved dimension is split in order, while resolved dimensions are flagged as finished in all descendants. By inductive assumption, $\Tm$ was most recently split in dimension $r$. The algorithm determines which $\nu_k$ has splitting dimension $j$ that comes the soonest after $r$ (computed cyclically). Thus for all dimensions between $r$ and $j$, neither of the given nodes splits, so it and its descendants all must have \textsf{isdone} set to TRUE in those dimensions, and property~(\ref{eq:merge-zones}) makes no requirement. Furthermore, the dimension $r_k$ does satisfy~(\ref{eq:merge-zones}) for $\nu_k$, and the same will be true for its children and the children of $\Tm$. All other dimensions will inherit~(\ref{eq:merge-zones}) from the parents.

\newcommand{\op}{\ensuremath{\circ}}
\begin{algorithm}[!h]
\caption{merge(\op,$\nu_1$,$\nu_2$,$\num$,$r$)}
\label{alg:merge}
\begin{algorithmic}
\IF{$\nu_1$ and $\nu_2$ are leaves}
\STATE \textsf{values}($\num$):= \textsf{interpolant}($\nu_1$) $\op$ \textsf{interpolant}($\nu_2$), evaluated on \prop{grid}{$T_\text{merge}$}
\ELSIF{$\nu_1$ is a leaf \AND $\nu_2$ is not a leaf}
\STATE split($\num$,\prop{splitdim}{$\nu_1$})
\STATE merge(\op,$\nu_1$,\child{0}($\nu_2$),\child{0}($\num$),\prop{splitdim}{$\nu_2$})
\STATE merge(\op,$\nu_1$,\child{1}($\nu_2$),\child{1}($\num$),\prop{splitdim}{$\nu_2$})
\ELSIF{$\nu_1$ is not a leaf \AND $\nu_2$ is a leaf}
\STATE split($\num$,\prop{splitdim}{$\nu_1$})
\STATE merge(\child{0}($\nu_1$),$\nu_2$,\child{0}($\num$),\prop{splitdim}{$\nu_1$})
\STATE merge(\child{1}($\nu_1$),$\nu_2$,\child{1}($\num$),\prop{splitdim}{$\nu_1$})
\ELSE
\IF{\prop{splitdim}{$\nu_1$}=\prop{splitdim}{$\nu_2$}}
\STATE split($\num$,\prop{splitdim}{$\nu_1$})
\STATE merge(\op,\child{0}($\nu_1$),\child{0}($\nu_2$),\child{0}($\num$),\prop{splitdim}{$\nu_1$})
\STATE merge(\op,\child{1}($\nu_1$),\child{1}($\nu_2$),\child{1}($\num$),\prop{splitdim}{$\nu_1$})
\ELSE
\STATE $r_1 = (\text{\prop{splitdim}{$\nu_1$}}-r-1) \mod d$
\STATE $r_2 = (\text{\prop{splitdim}{$\nu_2$}}-r-1) \mod d$
\STATE If $r_1>r_2$, swap $\nu_1$ and $\nu_2$
\STATE split($\Tm$,\prop{splitdim}{$\nu_1$})
\STATE merge(\op,\child{0}($\nu_1$),$\nu_2$,\child{0}($\num$),\prop{splitdim}{$\nu_1$})
\STATE merge(\op,\child{1}($\nu_1$),$\nu_2$,\child{1}($\num$), \prop{splitdim}{$\nu_1$})
\ENDIF
\ENDIF
\end{algorithmic}
\end{algorithm}

\end{appendices}

\end{document}